\UseRawInputEncoding
\documentclass[12pt]{article}
\usepackage[T1]{fontenc}
\usepackage[reqno,tbtags]{amsmath}
\usepackage{amssymb,amsthm}
\usepackage{xcolor}
\usepackage[margin=1.05in]{geometry}
\usepackage[colorlinks=true,linkcolor=black,citecolor=black,urlcolor=black]{hyperref}
\hypersetup{pdftitle={Neumann eigenvalues of isosceles triangles: monotonicity and asymptotics},pdfauthor={Guowei Dai, Yingxin Sun, Yong Zhang}}
\newtheorem{theorem}{Theorem}[section]
\newtheorem{lemma}[theorem]{Lemma}
\newtheorem{proposition}[theorem]{Proposition}
\newcommand{\dd}{\,\mathrm{d}}
\newcommand{\Dlt}{\Delta}
\newcommand{\Ex}{E_x}
\newcommand{\Ey}{E_y}
\newcommand{\WS}{\mathcal W}
\newcommand{\VA}{\mathcal V}
\newcommand{\norm}[1]{\left\lVert#1\right\rVert}
\newcommand{\abs}[1]{\left\lvert#1\right\rvert}
\newcommand{\rev}[1]{#1}
\numberwithin{equation}{section}

\begin{document}
\title{Neumann eigenvalues of isosceles triangles: monotonicity and asymptotics
\thanks{Research supported by NNSF of China (No. 12371110 and No. 12301133).}}
\author{
Guowei Dai\thanks{Corresponding author. School of Mathematical Sciences,
Dalian University of Technology, Dalian 116024, P.R. China.
E-mail: daiguowei@dlut.edu.cn.}
\and
Yingxin Sun\thanks{School of Mathematical Sciences,
Dalian University of Technology, Dalian 116024, P.R. China.
E-mail: sunyingxin2023@mail.dlut.edu.cn.}
\and
Yong Zhang\thanks{School of Mathematical Sciences,
Jiangsu University, Zhenjiang 212013, P.R. China.
E-mail: zhangyong@ujs.edu.cn.}}
\date{}
\maketitle

\begin{abstract}
We prove Laugesen and Siudeja's conjecture on Neumann eigenvalues of
isosceles triangles. After multiplication by the squared diameter, the
first positive symmetric eigenvalue increases strictly with the apex angle $\theta$, while the first antisymmetric eigenvalue decreases strictly up
to the equilateral triangle and increases strictly thereafter.
\rev{The proof combines a slice-average estimate of directional energies
with the explicit equilateral eigenfunction.} A quadratic
correction across slices also gives a two-term expansion for every
fixed positive spectral index $k$ as $\theta\downarrow0$, with relative
correction $\theta^2/6$ and remainder $O_k(\theta^4)$.
\end{abstract}

\medskip
\noindent\textit{Keywords.} Neumann eigenvalue; Isosceles triangle;
Monotonicity; Thin-domain asymptotics
\par\smallskip
\noindent\textit{2020 Mathematics Subject Classification.} 35P15; 35J25; 58J50

\section{Introduction and main results}

\quad\, We study how the Neumann eigenvalues of an isosceles triangle change
with its apex angle. The first positive eigenvalue
determines the optimal Poincar\'e constant and the lowest nonconstant
vibration frequency. Since dilation by $c>0$ multiplies an eigenvalue
by $c^{-2}$, the product $\mu D^2$, with $D$ the diameter, measures shape
independently of size. Laugesen and Siudeja established sharp Neumann
eigenvalue bounds for triangles~\cite{LSmin,LSmax}. In particular,
$\mu_2(T)D^2>j_{1,1}^2$ for every nondegenerate triangle, with equality
approached by thin isosceles triangles~\cite{LSmin}. Such an extremal bound
does not determine whether the eigenvalue increases along a family of
triangles. The present paper answers this question for the isosceles
family and quantifies the approach to its thin-triangle limit.
Here $j_{1,k}$ denotes the $k$th positive zero of the Bessel function $J_1$,
and the full Neumann spectrum is numbered, with multiplicities, as
$0=\mu_1<\mu_2\le\mu_3\le\cdots$.

Let $T_\theta$ be an isosceles triangle with apex angle
$\theta\in(0,\pi)$ between its equal sides, and let $D_\theta$ be its
diameter. All sides carry the Neumann condition. Reflection in the axis
splits the spectrum into symmetric and antisymmetric classes. Denote
their first positive eigenvalues by $\mu_S(T_\theta)$ and
$\mu_A(T_\theta)$, respectively. These are symmetry labels rather than
fixed positions in the full spectrum. The first positive eigenfunction
is symmetric for $\theta<\pi/3$ and antisymmetric for $\theta>\pi/3$;
at $\theta=\pi/3$, the eigenvalue has multiplicity two~\cite{LSmin}.
Motivated by numerical plots of both branches, Laugesen and Siudeja
formulated their diameter-normalized monotonicity as Conjecture 6.41
in Henrot's volume~\cite[p.~184]{LSchapter}. We prove the conjecture
with strict monotonicity.

\begin{theorem}\label{thm:main}
The function $\theta\mapsto\mu_S(T_\theta)D_\theta^2$ is strictly increasing
on $(0,\pi)$. The function $\theta\mapsto\mu_A(T_\theta)D_\theta^2$ is strictly
decreasing on $(0,\pi/3]$ and strictly increasing on $[\pi/3,\pi)$.
Its unique minimum is $16\pi^2/9$, attained at the equilateral triangle.
\end{theorem}

Together with the symmetry classification, Theorem~\ref{thm:main}
shows that $\mu_2(T_\theta)D_\theta^2$ increases strictly on the whole
interval $(0,\pi)$. It also describes each symmetry branch on the
range where that branch is not the first positive eigenvalue.
Thus it gives an ordering throughout the family, beyond the known
extremal bounds.

The main difficulty is the symmetric branch for $0<\theta<\pi/3$.
Mapping half of the triangle to a fixed reference triangle gives the
energy $E_x+tE_y$, where $t=\tan^2(\theta/2)$. For a normalized
minimizing eigenfunction, the derivative of the diameter-normalized
eigenvalue is $E_y-E_x/t^2$ at almost every $t$ in this range.
An upper bound for the eigenvalue alone cannot control this sign:
one needs a comparison of the two directional energies of the
minimizer itself.

Our main estimate provides this comparison through horizontal slice
averages. Differentiating the average produces a boundary term because
the slice length depends on height. Estimating that term gives
$E_x<t^2E_y$ for small $t$. A second variational comparison bounds the
horizontal energy by its explicit value at the equilateral triangle
and covers the remaining interval. The point $t=2/9$ is chosen in
the overlap of these estimates; $t=1/3$ is the geometric threshold
where the formula for the diameter changes. Section~\ref{sec:main-proof}
explains both choices. The method uses energy comparisons throughout,
so it does not require simplicity or differentiation of eigenfunctions
with respect to the angle.

\rev{The mixed-eigenvalue monotonicity of Gui, Hu, Li and
Zhang~\cite[Theorem 1.6]{GHLZ}, after rescaling their right triangles,
implies the decrease of the antisymmetric branch below the equilateral
angle. We give a short variational proof of this part and establish
the full diameter-normalized monotonicity in Theorem~\ref{thm:main}.}
\rev{Chen, Wu and Yao~\cite[Theorem 1.3]{CWY} prove that, on a class
of smooth symmetric domains, the first positive even Neumann
eigenvalue is simple in the even subspace and its eigenfunction is
monotone along the symmetry axis. Their Remark~5.3 also gives simplicity and spatial
monotonicity for the second Neumann eigenfunction on non-acute
triangles, including the right half-triangles used here. Spatial
monotonicity alone does not provide the quantitative comparison of
directional energies needed to control the variation of the
eigenvalue with the apex angle.}

The same slice decomposition also gives more precise information near
the degenerate triangle. The leading limit $j_{1,1}^2$ for the first
positive eigenvalue follows from Laugesen and Siudeja's thin-triangle
bounds~\cite{LSmin}; see also \cite[Proposition 6.42]{LSchapter}.
Henrot and Michetti~\cite[Lemma 3.5]{HM} establish convergence of each
fixed Neumann eigenvalue on thin planar domains to a one-dimensional
problem weighted by the width. For a triangle this weight is linear,
and the limiting problem is the radial Neumann equation on a disk.
We determine the next term and bound the remainder.

\begin{theorem}\label{thm:thin}
For each fixed integer $k\ge1$, as $\theta\downarrow0$,
\[
\mu_{k+1}(T_\theta)D_\theta^2
=j_{1,k}^2\left(1+\frac{\theta^2}{6}\right)+O_k(\theta^4).
\]
\rev{The implicit constant may depend on $k$.}
\end{theorem}

The relative quadratic correction is the same for every fixed spectral
index. To obtain it, we add a quadratic function of mean zero on each
slice, chosen by minimizing the next term in the energy. This gives an
upper bound that matches the lower bound from slice projection at the
first correction term. The explicit bounds in Proposition~\ref{prop:thin-bounds}
control the remainder as well as the coefficient. For $k=1$, the expansion
quantifies the rate at which thin isosceles triangles approach the sharp
lower bound $j_{1,1}^2$.

The paper is organized as follows. Section~\ref{sec:reference} formulates
the variational problem on a fixed triangle and explains the diameter
normalization. Section~\ref{sec:main-proof} proves Theorem~\ref{thm:main}
using slice estimates and directional energy comparisons.
Section~\ref{sec:thin} derives the thin-angle bounds and proves
Theorem~\ref{thm:thin}. Appendix~\ref{app:integrals} collects the
equilateral integration formulas used in the proof.

\section{A fixed reference triangle}\label{sec:reference}

\quad\, Normalize the equal sides to length one and choose coordinates so that
\[
T_\theta=\operatorname{conv}\left\{
\left(-\sin\frac\theta2,0\right),
\left(\sin\frac\theta2,0\right),
\left(0,\cos\frac\theta2\right)\right\},
\,\,\, t=\tan^2\frac\theta2.
\]
The map $(x,y)\mapsto\bigl(\sin(\theta/2)x,\cos(\theta/2)y\bigr)$ sends
\[
\Dlt=\{(x,y):x>0,\ y>0,\ x+y<1\}
\]
onto the right half of $T_\theta$. For real $v\in H^1(\Dlt)$, set
\[
M(v)=\int_\Dlt v^2\dd x\dd y,\,\,\,
\Ex(v)=\int_\Dlt v_x^2\dd x\dd y,\,\,\,
\Ey(v)=\int_\Dlt v_y^2\dd x\dd y.
\]
The fixed form domains for the two symmetry classes are
\[
\WS=\left\{v\in H^1(\Dlt):\int_\Dlt v\dd x\dd y=0\right\},
\,\,\,
\VA=\left\{v\in H^1(\Dlt):v\big|_{x=0}=0\right\}.
\]
The mean-zero condition removes the constant symmetric eigenfunction.
Antisymmetry gives a zero trace on the axis $x=0$; the other edges retain
their natural Neumann conditions. Define
\begin{equation}\label{eq:forms}
h_S(t)=\min_{v\in\WS\setminus\{0\}}
\frac{\Ex(v)+t\Ey(v)}{M(v)},\,\,\,
h_A(t)=\min_{v\in\VA\setminus\{0\}}
\frac{\Ex(v)+t\Ey(v)}{M(v)}.
\end{equation}
Poincar\'e inequalities and the compact embedding of $H^1(\Dlt)$ into
$L^2(\Dlt)$ give positive minima, attained by functions with $M(v)=1$.
Such a minimizer satisfies
\begin{equation}\label{eq:weak}
\int_\Dlt\bigl(v_x\eta_x+t v_y\eta_y\bigr)\dd x\dd y
=h(t)\int_\Dlt v\eta\dd x\dd y
\end{equation}
for every $\eta$ in the corresponding form domain.

Set $a=\sin(\theta/2)$ and $b=\cos(\theta/2)$, so that
$t=a^2/b^2$. Under the change of variables
$(X,Y)=(ax,by)$, write $u(ax,by)=v(x,y)$. On the right half
$T_\theta^+$ of the triangle, the Rayleigh quotient becomes
\[
\frac{\displaystyle\int_{T_\theta^+}
\left|\nabla u\right|^2\,\mathrm{d}X\,\mathrm{d}Y}
{\displaystyle\int_{T_\theta^+}u^2\,\mathrm{d}X\,\mathrm{d}Y}
=
\frac{a^{-2}E_x(v)+b^{-2}E_y(v)}{M(v)}
=
\frac{1+t}{t}\,
\frac{E_x(v)+tE_y(v)}{M(v)}.
\]
Taking the infimum over the corresponding symmetric and
antisymmetric variational spaces gives
\begin{equation}\label{eq:scale}
\mu_S(T_\theta)=\frac{1+t}{t}h_S(t),\,\,\,
\mu_A(T_\theta)=\frac{1+t}{t}h_A(t).
\end{equation}
The diameter is the larger of an equal side and the base. Thus
\[
D_\theta^2=\max\left\{1,\frac{4t}{1+t}\right\},\,\,\,
\mu(T_\theta)D_\theta^2=
\begin{cases}
\dfrac{1+t}{t}h(t),&0<t\le1/3,\\[4pt]
4h(t),&t\ge1/3,
\end{cases}
\]
where $h$ can stand for either $h_S$ or $h_A$. The expressions agree
at $t=1/3$.

Each function in \eqref{eq:forms} is the infimum of affine functions of
$t$, and is therefore concave and locally Lipschitz on $(0,+\infty)$.
If $v$ is any normalized minimizer at $t$, with
$e_x=\Ex(v)$ and $e_y=\Ey(v)$, the trial function $v$ gives
\[
h(s)\le e_x+s e_y=h(t)+(s-t)e_y.
\]
At a differentiability point of $h$, taking difference quotients from both
sides shows that $h'(t)=e_y$ for every normalized minimizer at that point.
Consequently, for $F(t)=(1+t)h(t)/t$,
\begin{equation}\label{eq:derivative}
F'(t)=-\frac{h(t)}{t^2}+\frac{1+t}{t}e_y
=e_y-\frac{e_x}{t^2}
\quad\text{for almost every }t>0.
\end{equation}
This gives the derivative needed below even if an eigenvalue is multiple.

\section{\texorpdfstring{Proof of Theorem~\ref{thm:main}}{Proof of Theorem 1.1}}\label{sec:main-proof}

\quad\, The distinction between $t<1/3$ and $t>1/3$ comes from the diameter
normalization. Both reference eigenvalues $h_S(t)$ and $h_A(t)$ increase
strictly for all $t>0$, as shown below. For $t\ge1/3$, the quantities in
Theorem~\ref{thm:main} are $4h_S(t)$ and $4h_A(t)$, so the same argument
settles both classes. For $t<1/3$, they are instead
\[
F_S(t)=\frac{1+t}{t}h_S(t),\,\,\,
F_A(t)=\frac{1+t}{t}h_A(t).
\]
The factor $(1+t)/t$ decreases, so increase of $h_S$ or $h_A$ alone
does not determine the sign of the normalized derivative.
By \eqref{eq:derivative}, the symmetric class requires $e_x<t^2e_y$,
whereas the antisymmetric class requires the reverse inequality.
We estimate the symmetric class by slice averages and equilateral data.
For the antisymmetric class, the zero trace on the axis supplies a
lower bound for $e_x$ through a Dirichlet--Neumann Poincar\'e inequality.

\subsection{Two energy estimates}

\quad\, The slice argument below uses the symmetric form domain $\WS$.
Subtracting a horizontal slice average preserves the zero mean on
$\Dlt$. It need not preserve a zero trace on $x=0$, so the resulting
remainder is not generally an admissible test function in $\VA$.
This is why that argument is used only for the symmetric eigenvalue.

\begin{lemma}\label{lem:slice}
Let $v$ be a normalized minimizer for $h=h_S(t)$, and put
$e_x=\Ex(v)$ and $e_y=\Ey(v)$. Then
\begin{equation}\label{eq:key}
\left(1-\frac{h}{\pi^2}-\frac t3\right)\sqrt{e_x}
\le\frac t{\sqrt3}\sqrt{e_y}.
\end{equation}
\end{lemma}

\begin{proof}
We first derive the slice identities for smooth functions and then
extend them to $H^1(\Dlt)$ before testing the eigenvalue equation.
Write $\ell(y)=1-y$ and decompose $v$ into its slice average and remainder:
\begin{equation}\label{eq:decomp}
\phi(y)=\frac1{\ell(y)}\int_0^{\ell(y)}v(x,y)\dd x,
\,\,\, w(x,y)=v(x,y)-\phi(y).
\end{equation}
On each horizontal slice, $\int_0^\ell w\dd x=0$ and $w_x=v_x$.
The Neumann Poincar\'e inequality on $(0,\ell)$ gives
\begin{equation}\label{eq:pw}
\norm{w}_{L^2(\Dlt)}^2
\le\int_0^1\frac{\ell^2}{\pi^2}
       \int_0^\ell w_x^2\dd x\dd y
\le\frac{e_x}{\pi^2}.
\end{equation}
Integration by parts on a slice gives
\[
\int_0^\ell xw_x\dd x
=\ell w(\ell,y)-\int_0^\ell w\dd x=\ell w(\ell,y).
\]
Since $\int_0^\ell x^2\dd x=\ell^3/3$, Cauchy--Schwarz yields
\begin{equation}\label{eq:trace}
\abs{w(\ell,y)}^2\le\frac\ell3\int_0^\ell w_x^2\dd x,
\,\,\,
\int_0^1\frac{\abs{w(\ell,y)}^2}{\ell}\dd y\le\frac{e_x}{3}.
\end{equation}
The slice endpoint moves with $y$, so differentiating an average
requires its boundary term. Differentiating $\int_0^\ell w\dd x=0$
and $\ell\phi=\int_0^\ell v\dd x$, with $\ell'=-1$, gives
\begin{equation}\label{eq:moving}
\int_0^\ell w_y\dd x=w(\ell,y),\,\,\,
\ell\phi'=\int_0^\ell v_y\dd x-w(\ell,y).
\end{equation}

To extend these calculations to $H^1$, let $Pv=\phi$ denote the
slice-average map. Jensen's inequality gives $M(Pv)\le M(v)$.
The second identity in \eqref{eq:moving}, the inequality
$\abs{a-b}^2\le2\abs a^2+2\abs b^2$, and \eqref{eq:trace} give
$\Ey(Pv)\le2e_y+2e_x/3$ for smooth $v$; also $\Ex(Pv)=0$.
Thus $P$ is bounded in $H^1$ and extends by smooth density to
$H^1(\Dlt)$. Applying \eqref{eq:trace} to differences of smooth
approximations gives convergence of the endpoint terms in
$L^2\bigl((0,1),\mathrm{d}y/\ell\bigr)$. Hence
\eqref{eq:pw}--\eqref{eq:moving} remain valid for $H^1$ functions,
including for slices whose lengths tend to zero.
For $E_\phi=\int_0^1\ell\abs{\phi'}^2\dd y$, the triangle inequality
in this weighted space, followed by Cauchy--Schwarz on each slice
and \eqref{eq:trace}, now gives the sharper bound
\begin{equation}\label{eq:projection}
\begin{split}
\sqrt{E_\phi}
&\le\left(\int_0^1\frac1\ell
  \abs{\int_0^\ell v_y\dd x}^2\dd y\right)^{1/2}
+\left(\int_0^1\frac{\abs{w(\ell,y)}^2}{\ell}\dd y\right)^{1/2}\\
&\le\sqrt{e_y}+\sqrt{e_x/3}.
\end{split}
\end{equation}

For the minimizing eigenfunction, both $\phi$ and $w$ have mean zero
on $\Dlt$, and slice orthogonality gives
$\int_\Dlt vw\dd x\dd y=\norm{w}_{L^2(\Dlt)}^2$.
Test \eqref{eq:weak} with $w$ and use \eqref{eq:moving}:
\[
e_x+t\int_\Dlt w_y^2\dd x\dd y
+t\int_0^1\phi'(y)w(\ell,y)\dd y
=h\norm{w}_{L^2(\Dlt)}^2.
\]
Dropping the nonnegative integral of $w_y^2$ and applying
\eqref{eq:pw} and \eqref{eq:trace} gives
\[
\left(1-\frac h{\pi^2}\right)e_x
\le t\left(\int_0^1\ell\abs{\phi'}^2\dd y\right)^{1/2}
       \left(\int_0^1\frac{\abs{w(\ell,y)}^2}{\ell}\dd y\right)^{1/2}
\le t\sqrt{E_\phi e_x/3}.
\]
If $e_x>0$, divide by $\sqrt{e_x}$ and apply \eqref{eq:projection}
to obtain \eqref{eq:key}. If $e_x=0$, it is immediate.
\end{proof}

Increasing $t$ makes vertical variation more expensive in the form
$\Ex+t\Ey$. The following comparison describes how the minimizing
energies change and will allow us to use the equilateral eigenfunction.

\begin{lemma}\label{lem:direction}
Let $0<t_1<t_2$, and choose any normalized minimizers of $h_S(t_1)$
and $h_S(t_2)$, with energies $e_{x,i}$ and $e_{y,i}$. Then
$e_{y,2}\le e_{y,1}$ and $e_{x,1}\le e_{x,2}$.
\end{lemma}

\begin{proof}
Let $v_i\in\WS$ be the chosen minimizer at $t_i$, with $M(v_i)=1$,
and write
\[
h_i=h_S(t_i)=e_{x,i}+t_i e_{y,i},\,\,\, i=1,2.
\]
The form domain $\WS$ does not depend on $t$. Thus $v_2$ is an
admissible trial function at $t_1$, and its Rayleigh quotient gives
\[
h_1\le e_{x,2}+t_1e_{y,2}
=h_2-(t_2-t_1)e_{y,2}.
\]
Similarly, using $v_1$ as a trial function at $t_2$ gives
\[
h_2\le e_{x,1}+t_2e_{y,1}
=h_1+(t_2-t_1)e_{y,1}.
\]
Rearranging these inequalities and dividing by the positive difference
$t_2-t_1$ gives
\[
e_{y,2}\le\frac{h_2-h_1}{t_2-t_1}\le e_{y,1},
\]
which proves the claimed comparison of vertical energies.
To compare the horizontal energies, return to the first trial inequality
and substitute $h_1=e_{x,1}+t_1e_{y,1}$:
\[
e_{x,1}-e_{x,2}\le t_1(e_{y,2}-e_{y,1})\le0.
\]
The last inequality uses $t_1>0$ and the vertical-energy comparison
just proved. Hence $e_{x,1}\le e_{x,2}$.
Only minimality and normalization were used, so the argument applies
to any choice of minimizers.
\end{proof}

\subsection{The monotonicity argument}

\quad\, Lemma~\ref{lem:slice} controls the horizontal energy for small $t$,
while Lemma~\ref{lem:direction} allows comparison with the equilateral
eigenfunction for the remaining symmetric range. We now combine these
estimates with the diameter normalization and a Poincar\'e bound for
the antisymmetric class.

\begin{proof}[Proof of Theorem~\ref{thm:main}]
\textbf{Step 1: both symmetry classes for $t\ge1/3$.}
For any $t>0$, a normalized minimizer in either class has $\Ey(v)>0$.
To prove this, suppose $v_y=0$, so that $v=f(x)$.
In the symmetric case, \eqref{eq:weak} holds for all $H^1$ test functions:
subtracting their mean changes neither side because $v$ has mean zero.
Testing with a function of $y$ and using $h>0$ gives
\[
\int_0^{1-y}f(x)\dd x=0\quad\text{for almost every }y\in(0,1).
\]
Differentiating this indefinite integral gives $f=0$, a contradiction.
In the antisymmetric case, $f(0)=0$.
For $\zeta\in C_c^\infty(0,1)$, take
$\eta(x,y)=x\zeta(y)\in\VA$ in \eqref{eq:weak}.
Since $\eta_x=\zeta(y)$, $\eta_y=x\zeta'(y)$, and $v_y=0$, we obtain
\[
\int_0^1\zeta(y)\int_0^{1-y}f'(x)\dd x\dd y
=h\int_0^1\zeta(y)\int_0^{1-y}xf(x)\dd x\dd y.
\]
The inner integral on the left is $f(1-y)-f(0)=f(1-y)$.
As this holds for every $\zeta$, we obtain, with $\ell=1-y$,
\[
f(\ell)=h\int_0^\ell xf(x)\dd x
\quad\text{for almost every }\ell\in(0,1).
\]
On every compact subinterval of $[0,1)$, the function $f$ has ordinary
$H^1$ regularity. Differentiating there gives $f'(\ell)=h\ell f(\ell)$.
Together with $f(0)=0$, this implies $f=0$ throughout $(0,1)$,
again contradicting normalization.

Now let $v$ be a normalized minimizer at $t_2>t_1>0$.
Since $\Ey(v)>0$, its Rayleigh quotient at $t_1$ gives
\[
h(t_1)\le h(t_2)-(t_2-t_1)\Ey(v)<h(t_2).
\]
Thus $h_S$ and $h_A$ both increase strictly on $(0,+\infty)$.
For $t\ge1/3$, the normalized eigenvalues are $4h_S(t)$ and $4h_A(t)$,
so both increase strictly in this range.

\textbf{Step 2: the symmetric class for $0<t\le2/9$.}
We need $e_x<t^2e_y$ to apply \eqref{eq:derivative}.
Functions independent of $x$ are natural trial functions for small $t$.
\rev{Set $r=1-y$ and write $\psi(y)=p(r)$, where $p$ is an even
polynomial with zero weighted mean. The conditions $p'(0)=p'(1)=0$
are not required for a Neumann trial function, but match the radial
modes in Section~\ref{sec:thin}. A nonconstant quadratic cannot satisfy
both conditions, so we take $p(r)=r^4+ar^2+b$.
Then $p'(1)=0$ gives $a=-2$, and $\int_0^1rp(r)\dd r=0$ gives
$b=2/3$. Thus}
\[
\rev{\psi(y)=p(1-y)=r^4-2r^2+\frac23.}
\]
The relevant integrals are
\begin{align*}
\int_\Dlt\psi\dd x\dd y
&=\int_0^1r\left(r^4-2r^2+\frac23\right)\dd r
=\frac16-\frac12+\frac13=0,\\
M(\psi)&=\int_0^1r\left(r^4-2r^2+\frac23\right)^2\dd r
=\frac2{45},\\
\Ey(\psi)&=16\int_0^1r^3\left(r^2-1\right)^2\dd r=\frac23,
\,\,\, \Ex(\psi)=0.
\end{align*}
Thus $h_S(t)\le15t$.
\rev{The choice is not unique. Within the family $r^6+ar^2+b$, the
same conditions give $p_6(r)=r^6-3r^2+5/4$, with weighted Rayleigh
quotient $3/(213/1120)=1120/71>15$. This weaker estimate does not
cover the chosen endpoint $t=2/9$ in the comparison below.
The quartic provides a simple bound that does.}

By Lemma~\ref{lem:slice}, the desired energy
comparison follows whenever
\[
1-\left(\frac{15}{\pi^2}+\frac13\right)t>\frac1{\sqrt3},
\,\,\,\text{or equivalently}\,\,\,
t<\tau_s:=\frac{1-1/\sqrt3}{15/\pi^2+1/3}.
\]
We choose $2/9$ as an endpoint within this range. Indeed,
$\pi^2>49/5$, $25/27-50/147-7/12=13/5292>0$, and $49/144>1/3$.
Hence, for $0<t\le2/9$,
\begin{equation}\label{eq:coefficient}
1-\frac{h_S(t)}{\pi^2}-\frac t3
\ge\frac{25}{27}-\frac{10}{3\pi^2}
>\frac7{12}>\frac1{\sqrt3}.
\end{equation}
By Step 1, $e_y>0$. If $e_x>0$, substituting this coefficient into
\eqref{eq:key} and squaring gives
\begin{equation}\label{eq:small-energy}
e_x<\frac{48}{49}t^2e_y<t^2e_y.
\end{equation}
The same strict inequalities hold immediately if $e_x=0$.

\textbf{Step 3: the symmetric class for $2/9\le t\le1/3$.}
Set
\[
\lambda_0=\frac{16\pi^2}{9},\,\,\,
h_0=h_S(1/3)=\frac{\lambda_0}{4}.
\]
The explicit equilateral spectrum gives a two-dimensional first positive
eigenspace, with a one-dimensional symmetric subspace. The eigenfunction
and its directional energies below are the data in
\cite[Section 4]{LSmin}, expressed on $\Dlt$. A symmetric eigenfunction is
\begin{equation}\label{eq:equilateral}
\phi_0(x,y)=\cos\left(\frac{2\pi x}{3}\right)
-2\cos\left(\frac{\pi x}{3}\right)\cos(\pi y).
\end{equation}
It satisfies $-\phi_{0,xx}-\phi_{0,yy}/3=h_0\phi_0$, with zero conormal
derivative on all three edges. On $x+y=1$, for example,
$\sin(\pi x/3+\pi y)=\sin(2\pi x/3)$ gives
$\phi_{0,x}+\phi_{0,y}/3=0$.
Integration as in Appendix~\ref{app:integrals} gives
\begin{equation}\label{eq:equilateral-integrals}
\begin{split}
&\int_\Dlt\phi_0\dd x\dd y=0,\,\,\, M(\phi_0)=\frac34,\\
&\Ex(\phi_0)=\frac{\pi^2}{6}-\frac{81}{64},\,\,\,
\Ey(\phi_0)=\frac{\pi^2}{2}+\frac{243}{64}.
\end{split}
\end{equation}
The normalized horizontal energy is
\[
e_0=\frac{\Ex(\phi_0)}{M(\phi_0)}
=\frac{\lambda_0}{8}-\frac{27}{16}.
\]
Lemma~\ref{lem:direction} gives $e_x\le e_0$ for $t<1/3$.
This also holds at $t=1/3$, since the symmetric eigenspace is
one-dimensional. Moreover, for $t\le1/3$,
\[
\Ex(v)+t\Ey(v)-3t\left(\Ex(v)+\frac13\Ey(v)\right)
=(1-3t)\Ex(v)\ge0.
\]
Taking infima yields $h_S(t)\ge3t h_0$. Since
\[
\frac{t h_S(t)}{1+t}-e_x=\frac{t^2e_y-e_x}{1+t},
\]
it remains to show that $3t^2h_0/(1+t)>e_0$.
Writing $q=e_0/h_0=1/2-243/(64\pi^2)$, this condition is
\[
t>\tau_e:=\frac{q+\sqrt{q^2+12q}}6.
\]
To check it for $t\ge2/9$, note that $t^2/(1+t)$ is increasing and
$3t^2h_0/(1+t)=\lambda_0/33$ at $t=2/9$.
Also, $e_0<\lambda_0/33$ is equivalent to $\pi^2<8019/800$,
which follows from $\pi^2<10$. Therefore
\begin{equation}\label{eq:middle}
\frac{t h_S(t)}{1+t}
\ge\frac{3t^2}{1+t}h_0
\ge\frac{\lambda_0}{33}>e_0\ge e_x,
\,\,\, 2/9\le t\le1/3.
\end{equation}
This proves $e_x<t^2e_y$ in the remaining symmetric range.
The checks in \eqref{eq:coefficient} and \eqref{eq:middle} also show
$\tau_e<2/9<\tau_s$: the two estimates overlap, and $2/9$ is a
convenient rational endpoint rather than a geometric or spectral transition.

By \eqref{eq:small-energy}, \eqref{eq:middle}, and \eqref{eq:derivative},
$F_S'(t)>0$ for almost every $t\in(0,1/3)$.
Local Lipschitz continuity gives absolute continuity on each compact
subinterval, so integration proves strict increase up to $t=1/3$.

\textbf{Step 4: the antisymmetric class for $0<t\le1/3$.}
Set $a=\pi^2/4$. On a horizontal slice of length $\ell\le1$, the zero
trace at $x=0$ and the Dirichlet--Neumann Poincar\'e inequality give
\[
\int_0^\ell v_x^2\dd x
\ge\frac{\pi^2}{4\ell^2}\int_0^\ell v^2\dd x
\ge a\int_0^\ell v^2\dd x,
\,\,\, v\in\VA.
\]
Thus $\Ex(v)\ge aM(v)$. For the trial function $v=x$,
$M(v)=1/12$, $\Ex(v)=1/2$, and $\Ey(v)=0$, giving $h_A(t)\le6$.
Every normalized minimizer consequently satisfies
\[
e_x\ge a,\,\,\, e_y\le\frac{6-a}{t}.
\]
For $0<t_1<t_2\le1/3$, use a minimizer at $t_1$ as a trial function
at $t_2$. Applying \eqref{eq:scale} and the preceding bounds yields
\begin{align*}
F_A(t_2)-F_A(t_1)
&\le(t_2-t_1)\left(e_y-\frac{e_x}{t_1t_2}\right)\\
&\le\frac{t_2-t_1}{t_1t_2}\bigl(t_2(6-a)-a\bigr)\\
&\le\frac{t_2-t_1}{t_1t_2}\left(2-\frac{\pi^2}{3}\right)<0.
\end{align*}
Hence $F_A$ decreases strictly up to $t=1/3$.

Finally, the normalization formulas agree at $t=1/3$, and
$t=\tan^2(\theta/2)$ increases strictly with $\theta$.
Steps 1--4 prove both monotonicity assertions in Theorem~\ref{thm:main}.
At the equilateral triangle the symmetric and antisymmetric modes have
the common normalized eigenvalue $16\pi^2/9$, which is therefore the
unique minimum of the antisymmetric branch.
\end{proof}

\section{Thin-angle asymptotics}\label{sec:thin}

\quad\,  Let $h_k(t)$ be the $k$th eigenvalue of the form $\Ex+t\Ey$ on the
zero-mean space $\WS$, counted with multiplicity, so that $h_1=h_S$.
Set
\[
\Lambda_k=j_{1,k}^2,\,\,\,
C_k=\frac{\Lambda_k}{5}+\frac{\Lambda_k^2}{45},\,\,\,
t_k=\min\left\{\frac12,\frac{\pi^2}{2\Lambda_k}\right\}.
\]

\begin{proposition}\label{prop:thin-bounds}
For $0<t\le t_k$,
\begin{equation}\label{eq:thin-bounds}
\Lambda_k\left(t-\frac{t^2}{3\left(1-\Lambda_k t/\pi^2\right)}\right)
\le h_k(t)
\le\Lambda_k\left(t-\frac{t^2}{3}\right)+C_k t^3.
\end{equation}
In particular, $h_k(t)=\Lambda_k t-\Lambda_k t^2/3+O_k(t^3)$.
\end{proposition}

Functions that are constant on each horizontal slice give the
leading term. We first identify their one-dimensional eigenvalue
problem. We then use slice averages for the lower bound and add a
quadratic correction across each slice for the upper bound.

\begin{proof}[Proof of Proposition~\ref{prop:thin-bounds}]
For functions on $(0,1)$, define
\[
M_F=\int_0^1rF(r)^2\dd r,\,\,\,
E_F=\int_0^1rF'(r)^2\dd r.
\]
The weight $r$ is the length of the slice at height $y=1-r$.
The relevant form space consists of $F\in H^1_{\mathrm{loc}}(0,1)$ with
$M_F+E_F<+\infty$
and $\int_0^1rF(r)\dd r=0$. Its positive variational eigenvalues are
$\Lambda_k$. To see this, identify $F$ with a radial function on the
unit disk: its squared norm and energy are $2\pi M_F$ and $2\pi E_F$.
The radial equation and outer boundary condition are
\[
-(rF')'=\Lambda rF,\,\,\, F'(1)=0.
\]
For $\Lambda>0$, the solutions are combinations of
$J_0(\sqrt\Lambda\,r)$ and $Y_0(\sqrt\Lambda\,r)$.
The second solution has derivative of order $1/r$ near zero and hence
infinite weighted energy. The finite-energy solution is therefore a
multiple of $J_0(\sqrt\Lambda\,r)$. Since $J_0'=-J_1$, the outer
boundary condition gives $\Lambda=\Lambda_k$.
Compactness and the min--max principle follow from the radial subspace
of $H^1$ on the disk.

Let $\mathcal B_k$ be the span of the first $k$ positive radial modes.
Choose these modes $F_1,\ldots,F_k$ orthonormal in
$L^2\bigl((0,1),r\dd r\bigr)$.
\rev{Each such mode has zero weighted mean: integrating its eigenvalue
equation and using $F_j'(1)=0$ and
$rF_j'(r)\to0$ as $r\downarrow0$ gives
$\Lambda_j\int_0^1rF_j(r)\dd r=0$.}
For $F=\sum_{j=1}^k c_jF_j$, orthogonality
and the eigenvalue equation give
\[
M_F=\sum_{j=1}^k c_j^2,
\,\,\, E_F=\sum_{j=1}^k\Lambda_jc_j^2\le\Lambda_kM_F.
\]
The lifted function $v(x,y)=F(1-y)$ satisfies
\[
\int_\Dlt v\dd x\dd y=\int_0^1rF(r)\dd r=0,
\,\,\, M(v)=M_F,\,\,\, \Ex(v)=0,\,\,\, \Ey(v)=E_F.
\]
Thus $\mathcal U_k=\{F(1-y):F\in\mathcal B_k\}$ is contained in
$\WS$. The lift is injective, so $\dim\mathcal U_k=k$. Hence
\rev{\begin{equation}\label{eq:unadjusted}
\begin{split}
h_k(t)
&=\min_{\substack{\mathcal U\subset\WS\\\dim\mathcal U=k}}
  \max_{v\in\mathcal U\setminus\{0\}}
  \frac{\Ex(v)+t\Ey(v)}{M(v)}\\
&\le\max_{F\in\mathcal B_k\setminus\{0\}}\frac{tE_F}{M_F}
=t\Lambda_k.
\end{split}
\end{equation}}

We now project onto slice averages to obtain the lower bound.
Write $H=h_k(t)$, and let $\mathcal E_k(t)$ be the span of $k$
orthonormal eigenfunctions for the first $k$ eigenvalues on $\WS$.
For $v\in\mathcal E_k(t)$,
\[
\Ex(v)+t\Ey(v)\le HM(v).
\]
Decompose $v=\phi+w$ as in \eqref{eq:decomp}. Then
$M(v)=M(\phi)+M(w)$ and $\Ex(v)=\Ex(w)$. Expanding the vertical
energy and using \eqref{eq:moving} gives
\[
\Ey(v)=E_\phi+2\int_0^1\phi'(y)w(\ell,y)\dd y+\Ey(w).
\]
By \eqref{eq:unadjusted} and the definition of $t_k$, $H\le\pi^2/2$,
so $1-H/\pi^2\ge1/2$. The bounds \eqref{eq:pw} and \eqref{eq:trace}
apply to every $v\in H^1$. Dropping $\Ey(w)$ and completing a square
in $\sqrt{\Ex(w)}$ therefore yields
\begin{align*}
\Ex(v)+t\Ey(v)-HM(v)
&\ge\left(1-\frac H{\pi^2}\right)\Ex(w)
 -2t\sqrt{E_\phi\Ex(w)/3}+tE_\phi-HM(\phi)\\
&\ge\left(t-\frac{t^2}{3\left(1-H/\pi^2\right)}\right)E_\phi
 -HM(\phi).
\end{align*}
Since $t\le1/2$, the coefficient of $E_\phi$ is positive:
\[
t-\frac{t^2}{3\left(1-H/\pi^2\right)}
\ge t\left(1-\frac{2t}{3}\right)\ge\frac{2t}{3}>0.
\]
The average map is injective on $\mathcal E_k(t)$. Indeed, if $\phi=0$,
then \eqref{eq:pw} implies $\Ex(v)\ge\pi^2 M(v)$, whereas
$\Ex(v)+t\Ey(v)\le HM(v)$ with $H<\pi^2$. Thus $v=0$.
Its image is consequently a $k$-dimensional weighted zero-mean space.
\rev{More precisely, setting $F(r)=\phi(1-r)$ identifies this image
with a subspace of the one-dimensional form domain, with
$M_F=M(\phi)$ and $E_F=E_\phi$.}
Since $\Ex(v)+t\Ey(v)-HM(v)\le0$,
the maximum of $E_\phi/M(\phi)$ over this image is at most
$H/\bigl(t-t^2/(3(1-H/\pi^2))\bigr)$.
The weighted min--max principle therefore gives
\[
H\ge\Lambda_k\left(t-\frac{t^2}{3\left(1-H/\pi^2\right)}\right)
\ge\Lambda_k\left(t-\frac{t^2}{3\left(1-\Lambda_k t/\pi^2\right)}\right).
\]
This proves the lower bound in \eqref{eq:thin-bounds}.

For the upper bound, we correct the trial function $F(1-y)$
to allow small horizontal variation. Fix $F\in\mathcal B_k$.
To find its first correction, write $r=1-y$ and try
$v=F(r)+t\chi(x,y)$, where $\chi\in H^1(\Dlt)$ has zero mean
on every slice. The coefficient of $t^2$ in the energy is
\[
Q_F(\chi)=\Ex(\chi)
+2\int_\Dlt\partial_yF(r)\,\chi_y\dd x\dd y.
\]
To compute the mixed term, differentiate the zero slice mean,
using $r=1-y$:
\[
0=\frac{\mathrm{d}}{\mathrm{d}y}\int_0^{1-y}\chi(x,y)\dd x
=\int_0^r\chi_y(x,y)\dd x-\chi(r,y).
\]
Integration by parts in $x$ also gives
\[
\int_0^r x\chi_x(x,y)\dd x
=r\chi(r,y)-\int_0^r\chi(x,y)\dd x=r\chi(r,y).
\]
These identities hold in the weak sense for $\chi\in H^1(\Dlt)$,
as in \eqref{eq:moving}. Since $\partial_yF(r)=-F'(r)$, they yield
\begin{align*}
\int_\Dlt\partial_yF(r)\,\chi_y\dd x\dd y
&=-\int_0^1F'(r)\chi(r,1-r)\dd r,\\
\chi(r,1-r)&=\frac1r\int_0^r x\chi_x(x,1-r)\dd x.
\end{align*}
Substituting the endpoint formula and completing the square gives
\begin{align*}
Q_F(\chi)
&=\int_0^1\int_0^r
\left(\chi_x(x,1-r)^2
-2\frac{x}{r}F'(r)\chi_x(x,1-r)\right)\dd x\dd r\\
&=\int_0^1\int_0^r
\left(\chi_x(x,1-r)-\frac{x}{r}F'(r)\right)^2\dd x\dd r
-\int_0^1\frac{F'(r)^2}{r^2}\int_0^r x^2\dd x\dd r\\
&=\int_0^1\int_0^r
\left(\chi_x(x,1-r)-\frac{x}{r}F'(r)\right)^2\dd x\dd r
-\frac{E_F}{3},
\end{align*}
because $\int_0^r x^2\dd x=r^3/3$ and
$E_F=\int_0^1rF'(r)^2\dd r$.

The minimum is $-E_F/3$, attained when $\chi_x=xF'(r)/r$.
This accounts for the coefficient $-1/3$ in
Proposition~\ref{prop:thin-bounds}. Integrating in $x$ and imposing
zero slice mean gives
\begin{equation}\label{eq:corrector-definition}
\chi_F(x,y)=\frac{F'(r)}{2r}\left(x^2-\frac{r^2}{3}\right),
\,\,\, v_F(x,y)=F(r)+t\chi_F(x,y).
\end{equation}

The ratio $F'(r)/r$
is regular at zero because every function in this finite Bessel span
is an even analytic function of $r$. Hence $v_F\in H^1(\Dlt)$.
Its slice average is $F(r)$, so $v_F\in\WS$ and the map
$F\mapsto v_F$ is injective.

For its mass, use slice orthogonality
and substitute $z=x/r$; the integral
$\int_0^1(z^2-1/3)^2\dd z=4/45$ gives $M(\chi_F)$ below.
The energy identities follow from $\partial_x\chi_F=xF'(r)/r$
and the moving-endpoint calculation above. Thus
\begin{equation}\label{eq:corrector}
\begin{split}
M(v_F)&=M_F+t^2M(\chi_F),\,\,\,
M(\chi_F)=\frac1{45}\int_0^1r^3F'(r)^2\dd r,\\
\Ex(\chi_F)&=\frac{E_F}{3},\,\,\,
\int_\Dlt\partial_yF(r)\,\partial_y\chi_F\dd x\dd y
=-\frac{E_F}{3}.
\end{split}
\end{equation}

\rev{Since $r^3\le r$ on $(0,1)$,
$M(\chi_F)\le E_F/45\le\Lambda_k M_F/45$.
Thus the relative mass correction is $O_k(t^2)$ and, because the
leading energy is of order $t$, affects the Rayleigh quotient only
at order $t^3$.}

To control the remaining energy, put $\mathcal LF=-F''-F'/r$.
Differentiating \eqref{eq:corrector-definition} at fixed $x$ gives
\[
\partial_y\chi_F
=-\frac12\left(rF''-F'\right)\left(z^2-\frac13\right)
+\frac{F'}3,\,\,\, z=\frac xr.
\]
Since $\int_0^1(z^2-1/3)\dd z=0$, integration in $z$ yields
\[
\Ey(\chi_F)=\int_0^1r\left(
\frac{(rF''-F')^2}{45}+\frac{F'^2}{9}\right)\dd r.
\]
Use $rF''-F'=-r\mathcal LF-2F'$ to expand the square.
Regularity at zero and $F'(1)=0$ make the integrated boundary term
vanish, so the cross term satisfies
\[
\int_0^1r^2(\mathcal LF)F'\dd r
=-\int_0^1\bigl(r^2F''F'+rF'^2\bigr)\dd r
=-\frac12\left[r^2F'^2\right]_0^1=0.
\]

To bound the remaining terms, expand $F$ in the weighted orthonormal
radial modes. Their eigenvalues are at most $\Lambda_k$, so
\[
E_F\le\Lambda_kM_F,\,\,\,
\int_0^1r^3(\mathcal LF)^2\dd r
\le\int_0^1r(\mathcal LF)^2\dd r\le\Lambda_k^2M_F.
\]
Substitution into the expanded square gives
\begin{equation}\label{eq:corrector-y}
\Ey(\chi_F)=\frac15E_F
+\frac1{45}\int_0^1r^3(\mathcal LF)^2\dd r\le C_k M_F.
\end{equation}

Because $t-t^2/3>0$ for $t\le t_k$, we can apply
$E_F\le\Lambda_kM_F$ together with \eqref{eq:corrector} and
\eqref{eq:corrector-y} to obtain
\[
\Ex(v_F)+t\Ey(v_F)
=\left(t-\frac{t^2}{3}\right)E_F+t^3\Ey(\chi_F)
\le\left[\Lambda_k\left(t-\frac{t^2}{3}\right)+C_k t^3\right]M_F.
\]
Since $M(v_F)\ge M_F$,
the min--max principle on the $k$-dimensional space
$\{v_F:F\in\mathcal B_k\}$ proves the upper bound in
\eqref{eq:thin-bounds}\rev{.}
\rev{Since $\Lambda_k t/\pi^2\le1/2$, the two bounds also give
\[
-\frac{2\Lambda_k^2}{3\pi^2}t^3
\le h_k(t)-\Lambda_k\left(t-\frac{t^2}{3}\right)
\le C_k t^3,
\]
which proves the stated expansion.}
\renewcommand{\qedsymbol}{$\square$}
\end{proof}

Proposition~\ref{prop:thin-bounds} gives the expansion within the
symmetric class. To deduce Theorem~\ref{thm:thin}, we must first identify
these eigenvalues with the first positive eigenvalues of the full triangle
for small $t$, and then express the expansion in terms of the angle
$\theta$.

\begin{proof}[Proof of Theorem~\ref{thm:thin}]
Reflection in the axis commutes with the Neumann Laplacian, so the
full spectrum is the union, with multiplicities, of the two symmetry
spectra. The lowest antisymmetric eigenvalue satisfies
\[
\mu_A(T_\theta)\ge\frac{(1+t)\pi^2}{4t}\longrightarrow+\infty
\quad\text{as }t\downarrow0,
\]
whereas \eqref{eq:unadjusted} bounds the $k$th positive symmetric
eigenvalue by $(1+t)\Lambda_k$. Therefore, if
$t<\pi^2/(4\Lambda_k)$, the first $k$ positive eigenvalues of the full
triangle all belong to the symmetric class. For sufficiently small $t$,
$D_\theta=1$, and Proposition~\ref{prop:thin-bounds} implies
\[
\mu_{k+1}(T_\theta)D_\theta^2
=\frac{1+t}{t}h_k(t)
=\Lambda_k\left(1+\frac23t\right)+O_k(t^2).
\]
The energy correction $-t^2/3$ and the scaling factor $(1+t)/t$
have thus produced the coefficient $2/3$ of $t$.
Substituting $t=\tan^2(\theta/2)=\theta^2/4+O(\theta^4)$ gives the
coefficient $1/6$ of $\theta^2$ and the stated $O_k(\theta^4)$ remainder.
\end{proof}

\appendix
\section{Equilateral integration formulas}\label{app:integrals}

\quad\, The data in \eqref{eq:equilateral-integrals} can be checked using
\[
C(a,b)=\int_0^1\int_0^{1-x}
\cos(a\pi x)\cos(b\pi y)\dd y\dd x.
\]
Integration in $y$ followed by a product-to-sum identity gives
\[
C(a,b)=\frac{\cos(b\pi)-\cos(a\pi)}{\pi^2(a^2-b^2)}
\quad\text{if }a^2\ne b^2,
\]
with the values at $a=\pm b$ obtained by continuity. In particular,
$C(0,0)=1/2$. The values needed here are
\[
\begin{array}{c|c@{\,\,\,}c|c}
(a,b)&C(a,b)&(a,b)&C(a,b)\\\hline
(0,0)&1/2&(0,2)&0\\[2pt]
(2/3,0)&27/(8\pi^2)&(2/3,2)&-27/(64\pi^2)\\[2pt]
(4/3,0)&27/(32\pi^2)&(1,1)&0\\[2pt]
(1/3,1)&27/(16\pi^2)&&
\end{array}
\]
From \eqref{eq:equilateral},
\[
\int_\Dlt\phi_0\dd x\dd y=C(2/3,0)-2C(1/3,1)=0.
\]
Expanding the square and applying product-to-sum identities gives
\begin{align*}
M(\phi_0)
&=\frac32C(0,0)+\frac12C(4/3,0)
-2C(1,1)-2C(1/3,1)\\
&\hspace{1.5em}+C(2/3,0)+C(0,2)+C(2/3,2)=\frac34.
\end{align*}
Similarly,
\begin{align*}
\phi_{0,x}
&=\frac{2\pi}{3}\left(-\sin\frac{2\pi x}{3}
+\sin\frac{\pi x}{3}\cos(\pi y)\right),\\
\phi_{0,y}&=2\pi\cos\frac{\pi x}{3}\sin(\pi y).
\end{align*}
Therefore
\begin{align*}
\Ex(\phi_0)
&=\frac{4\pi^2}{9}\left(\frac34C(0,0)-\frac12C(4/3,0)
-C(1/3,1)+C(1,1)\right.\\
&\hspace{6em}\left.-\frac14C(2/3,0)+\frac14C(0,2)
-\frac14C(2/3,2)\right)\\
&=\frac{4\pi^2}{9}\left(\frac38-\frac{729}{256\pi^2}\right)
=\frac{\pi^2}{6}-\frac{81}{64},\\[4pt]
\Ey(\phi_0)
&=\pi^2\bigl(C(0,0)+C(2/3,0)-C(0,2)-C(2/3,2)\bigr)\\
&=\frac{\pi^2}{2}+\frac{243}{64}.
\end{align*}
As a check, these values satisfy
$\Ex(\phi_0)+\Ey(\phi_0)/3=(4\pi^2/9)M(\phi_0)$.
\\ \\
\textbf{The conflicts of interest statement and Data Availability statement.}
\bigskip\\
\indent There is not any conflict of interest.
Data sharing not applicable to this article as no datasets were generated or analysed during the current study.
\\\\
\textbf{Declaration of AI-assisted writing.}
\bigskip\\
\indent During the preparation of this work, the authors used ChatGPT and DeepSeek for language polishing, grammar checking and refining the exposition of certain technical arguments. After using these tools, the authors carefully reviewed and edited the content as needed and take full responsibility for the final version of this manuscript.


\begin{thebibliography}{9}

\bibitem{LSmin}
R. S. Laugesen and B. A. Siudeja,
\emph{Minimizing Neumann fundamental tones of triangles: an optimal
Poincar\'e inequality}, J. Differential Equations \textbf{249} (2010),
118--135.

\bibitem{LSmax}
R. S. Laugesen and B. A. Siudeja,
\emph{Maximizing Neumann fundamental tones of triangles},
J. Math. Phys. \textbf{50} (2009), 112903.

\bibitem{LSchapter}
R. S. Laugesen and B. A. Siudeja,
\emph{Triangles and other special domains},
in A. Henrot (ed.), \emph{Shape Optimization and Spectral Theory},
De Gruyter Open, 2017, pp. 149--200.

\bibitem{GHLZ}
C. Gui, Y. Hu, Q. Li and C. Zhang,
\emph{Mixed torsion on right triangles and the P\'olya--Szeg\H{o}
monotonicity problem for regular polygons}, preprint, 2026.
\rev{\href{https://arxiv.org/abs/2606.13448}{arXiv:2606.13448}.}

\bibitem{CWY}
H.-B. Chen, K. Wu and R. Yao,
\emph{Monotone properties of the second even Neumann eigenfunction in
symmetric domains}, Ann. Mat. Pura Appl. \textbf{205} (2026), 635--657.

\bibitem{HM}
A. Henrot and M. Michetti,
\emph{A comparison between Neumann and Steklov eigenvalues},
J. Spectr. Theory \textbf{12} (2022), no. 4, 1405--1442.

\end{thebibliography}
\end{document}